\documentclass[12pt]{article}
\usepackage[top=2.45cm, bottom=2.45cm, left=2.25cm, right=2.25cm]{geometry}
\usepackage{amsmath,amssymb,amsthm,mathtools}
\usepackage{microtype}
\usepackage{enumitem}
\usepackage{booktabs}
\usepackage[hidelinks]{hyperref}
\usepackage[nameinlink,capitalize]{cleveref}
\usepackage{xcolor}

\newtheorem{theorem}{Theorem}[section]
\newtheorem{lemma}[theorem]{Lemma}

\theoremstyle{definition}

\newtheorem{problem}[theorem]{Problem}
\theoremstyle{remark}

\newcommand{\ceil}[1]{\left\lceil #1\right\rceil}
\newcommand{\floor}[1]{\left\lfloor #1\right\rfloor}

\begin{document}
\title{Internal partitions of $5$-regular and $8$-regular graphs}
\author{Siwei Lin, Qinghou Zeng\footnote{Research partially supported by National Key R\&D Program of China (Grant No. 2023YFA1010202) and National Natural Science Foundation of China (Grant No. 12371342). Email: zengqh@fzu.edu.cn (Corresponding author)}\\
{\small Center for Discrete Mathematics, Fuzhou University, Fujian, 350003, China}}
\date{}
\maketitle

\begin{abstract}
An \emph{internal partition} of a graph is a partition of its vertex
set into two nonempty parts such that every vertex has at least half of
its neighbors in its own part. DeVos asked whether, for every fixed positive integer $r$, all but finitely many $r$-regular graphs admit an internal partition. This was previously known to hold only for
$r\in\{3,4,6\}$. In this paper, we confirm this for $r\in\{5,8\}$ by proving that every $r$-regular graph on at least $10^7$ vertices has an internal partition.
\end{abstract}

\section{Introduction}

All graphs considered in this paper are finite, simple and undirected.
A partition $V(G)=A\cup B$ into two nonempty sets is called an
\emph{internal partition}, or a \emph{friendly partition}, if every
vertex has at least as many neighbours in its own part as in the other
part. Equivalently,
$d_A(v)\ge d_B(v)$ for every $v\in A$, and
$d_B(v)\ge d_A(v)$ for every $v\in B$.
For an $r$-regular graph, this condition becomes $\delta(G[A]),\delta(G[B])\ge \ceil{r/2}.$
The same or closely related notion has also been studied under the name
of a satisfactory partition. Its algorithmic and complexity aspects
were considered by Gerber and Kobler~\cite{GK2000} and by Bazgan,
Tuza and Vanderpooten~\cite{BTV2003,BTV2006}.

Internal partitions belong to the broader theory of graph
decompositions under degree constraints. Thomassen~\cite{Thomassen1983}
asked for conditions ensuring that the vertex set of a graph can be
partitioned into two parts with prescribed lower bounds on the
corresponding internal degrees. Stiebitz~\cite{Stiebitz1996} proved that
if $a,b:V(G)\to\mathbb Z_{\ge0}$ satisfy
$d_G(v)\ge a(v)+b(v)+1$ for every vertex $v$, then there is a partition
$V(G)=A\cup B$ such that
$d_A(v)\ge a(v)$ for every $v\in A$ and
$d_B(v)\ge b(v)$ for every $v\in B$.
Kaneko~\cite{Kaneko1998} showed that the additional one in this degree
condition can be omitted for triangle-free graphs. Diwan~\cite{Diwan2000}
obtained a still stronger statement for graphs of girth at least five,
and Ma and Yang~\cite{MY2019} extended the relevant decomposition result
to $C_4$-free graphs. More recent variants under linear minimum-degree
and average-degree assumptions were studied by Ma and Wu~\cite{MW2024}
and by Wang and Wu~\cite{WW2024}.

These general decomposition theorems do not settle the internal
partition problem for arbitrary regular graphs. Indeed, put
$k=\ceil{r/2}$. To obtain an internal partition, one needs internal
degree at least $k$ in both parts. Stiebitz's theorem would require
$r\ge2k+1$, whereas $r=2k$ when $r$ is even and $r=2k-1$ when $r$ is
odd. Thus the regular case lies just beyond the range of the general
degree-splitting theorem. At the same time, complete graphs show that
exceptions are unavoidable: $K_{r+1}$ has no internal partition.
For odd $r$, the complete bipartite graph $K_{r,r}$ is another natural
example without an internal partition.

Motivated by these examples, DeVos~\cite{DeVos2009} posed the following
problem.

\begin{problem}\label{DeVos-Internal}
For every fixed positive integer $r$, do all but finitely many
$r$-regular graphs have an internal partition?
\end{problem}

The problem has an affirmative answer for several small values of $r$.
Shafique and Dutton~\cite{SD2002} proved the required finite-exception
statement for $r\in\{3,4\}$. In these cases the exceptional graphs
are $K_4$ and $K_{3,3}$ for $r=3$, and $K_5$ for $r=4$.
Ban and Linial~\cite{BL2016} treated  $r=6$ and proved that every
$6$-regular graph on at least $12$ vertices has an internal partition;
the order bound is sharp. These results establish DeVos's conjectured
finite-exception property for $r\in\{3,4,6\}$, but they do not provide
a general method for the remaining  $r$.

Ban and Linial~\cite{BL2016} introduced a useful reformulation in terms
of cohesive sets. A nonempty set $U\subseteq V(G)$ is called
\emph{$k$-cohesive} if $\delta(G[U])\ge k$. They proved that every
$r$-regular graph has a $\ceil{d/2}$-cohesive set of order at most
$\ceil{n/2}$ when $r$ is even, and of order at most $n/2+1$ when $r$
is odd. If an $r$-regular graph contains two disjoint
$\ceil{r/2}$-cohesive sets, then these sets can be extended to the two
parts of an internal partition. Thus DeVos's problem can be viewed as
the problem of finding two disjoint cohesive sets of the required
minimum degree.

For $5$-regular graphs, B\"arnkopf, Nagy and
Paulovics~\cite{BNP2024} obtained substantial progress in this
direction. They proved that every $n$-vertex $5$-regular graph contains
two $3$-cohesive sets whose intersection has order at most $n/4+1$,
and improved the asymptotic coefficient to $0.2456+o(1)$. They also
determined the $5$-regular Cayley graphs over finite Abelian groups
without an internal partition. These results show that two cohesive
sets with relatively small intersection can always be found, but the
zero-intersection conclusion needed for an internal partition remained
out of reach.

In this paper, we solve Problem \ref{DeVos-Internal} for $r\in\{5,8\}$.

\begin{theorem}\label{thm:main}
For $r\in\{5,8\}$, every $r$-regular graph on at least $10^7$ vertices has an internal partition.
\end{theorem}

A nonempty set $S\subseteq V(G)$ is \emph{$k$-cohesive} if
$G[S]$ has minimum degree at least $k$. A \emph{bisection} is a
partition $(A,B)$ satisfying $\bigl||A|-|B|\bigr|\le1$. It is
\emph{minimum} if the number of edges between its two parts is minimum
among all bisections of $G$. We establish the following stronger
statement.

\begin{theorem}\label{thm:structure}
Let $r\in\{5,8\}$, and let $G$ be an $r$-regular graph on at least $10^7$ vertices. Each part of every minimum bisection of $G$
contains a $\ceil{r/2}$-cohesive set.
\end{theorem}

Two disjoint $\ceil{r/2}$-cohesive sets can be extended to an internal
partition. We include the short proof in Section~\ref{sec:exchange}.
Consequently, Theorem~\ref{thm:structure} implies
Theorem~\ref{thm:main}. The resulting internal partition need not be a
bisection, and the threshold $10^7$ is not claimed to be optimal. 


\noindent{\bf Notation.}
For a graph $G$ and a set $S\subseteq V(G)$, write
$e_G(S)=|E(G[S])|$ and $e(G)=|E(G)|$. For disjoint vertex sets $S,T$,
let $e_G(S,T)$ be the number of edges with one end in each set. We write
$d_G(v)$ for the degree of $v$ in $G$. When the ambient graph is $G$,
put $d_S(v)=|N_G(v)\cap S|$; in particular,
$e_G(\{v\},S)=d_S(v)$ when $v\notin S$.

\section{Minimum bisections and local inequalities}
\label{sec:exchange}

We first record the connection between disjoint cohesive sets and internal partitions.

\begin{lemma}[B\"arnkopf, Nagy and Paulovics \cite{BNP2024}]\label{lem:extension}
Let $G$ be an $r$-regular graph. If $G$ contains two disjoint $\ceil{r/2}$-cohesive sets, then $G$ has an internal partition.
\end{lemma}

\begin{proof}[\bf Proof.]
Let $U$ and $W$ be two such sets. Let us add the vertices, which have at least $\ceil{r/2}$ neighbours in $A$, from the complement of $U\cup W$ one by one to $U$.
After the procedure stops, we add the remaining vertices to $W$ and the resulting partition is internal.
\end{proof}

Recall that a graph is \emph{$d$-degenerate} if each of its nonempty
induced subgraphs has a vertex of degree at most $d$. A graph has no
$k$-cohesive set if and only if it is $(k-1)$-degenerate. It is easy to obtain that any $d$-degenerate graph $G$ on $n\ge d$ vertices satisfies
\begin{equation*}
 e(G)\le d(n-d)+\binom d2=dn-\binom{d+1}{2}.
\end{equation*}

The next key lemma translates the minimality of a bisection into the local edge inequality used throughout the proof.

\begin{lemma}\label{lem:exchange}
Let $r\in\{5,8\}$, let $G$ be an $r$-regular graph on at least $40$ vertices, and let $(A,B)$ be a minimum bisection. Suppose that $B$ contains no $\ceil{r/2}$-cohesive set. Let $a=|A|$, $b=|B|$, $p=\floor{r/2}+1$ and $N=a+1$.
There is an ordering $v_1,\ldots,v_b$ of $B$ such that, for $H=G[A\cup\{v_1\}]$, the following statements hold: 
\begin{enumerate}[label=(\roman*)]
    \item $2e(H)\le(p+1)N$;
    \item For every vertex $v\in V(H)$, $p\le d_H(v)\le r\le p+3$;
    \item Define $X_\ell=\cup_{i=2}^{\ell}(N_{G}(v_i)\cap V(H))$ for each $2\le \ell\le b$. 
        Then $|X_\ell|\le r(\ell-1)$ and for every nonempty set $S\subseteq V(H)\setminus X_\ell$ with $|S|\le \ell$,
        \begin{equation}\label{eq:local-exchange}
             e_H(S)\le\sum_{v\in S}(d_H(v)-p).
        \end{equation}
\end{enumerate}
\end{lemma}

\begin{proof}[\bf Proof.]
Write $c=e_G(A,B)$. 
If $u\in A$ and $v\in B$ are exchanged, the change in the cut size is $2\left(d_A(x)+d_B(y)-r+\mathbf 1_{xy\in E(G)}\right)$.
By the minimality of the bisection, for $u\in A,\ v\in B$, we have
\begin{equation}\label{eq:pair-exchange}
 d_A(u)+d_B(v)+\mathbf 1_{uv\in E(G)}\ge r.
\end{equation}
Since $B$ contains no $\ceil{r/2}$-cohesive set, $G[B]$ is $(\ceil{r/2}-1)$-degenerate.
Therefore
\begin{equation}\label{eq:cut-lower}
 e_G(A,B)
 =rb-2e_G(B)\ge (r-2\ceil{r/2})b+2b+\ceil{r/2}(\ceil{r/2}-1)
\end{equation}
Let $\delta_B$ be the minimum degree of $G[B]$, and choose $v\in B$ with
$d_B(v)=\delta_B$. From~\eqref{eq:pair-exchange}, for every $u\in A$,
\[
 d_B(u)=r-d_A(u)\le \delta_B+\mathbf 1_{uv\in E(G)}.
\]
Summing over $A$ yields
\begin{equation}\label{eq:cut-upper}
 e_G(A,B)\le a\delta_B+r-\delta_B.
\end{equation}
If $\delta_B\le\ceil{r/2}-2$, then the inequality~\eqref{eq:cut-upper}, together with $|a-b|\le1$, gives $e_G(A,B)\le (\ceil{r/2}-2)b+r$, contradicting~\eqref{eq:cut-lower} for $r=5,8$. Hence $\delta_B=\ceil{r/2}-1$. 

Now choose $v_1\in B$ with $d_B(v_1)=\delta_B$.
For $H=G[A\cup\{v_1\}]$, equation~\eqref{eq:pair-exchange} shows that every $u\in A$ satisfies
\[
 d_H(u)=d_A(u)+\mathbf 1_{uv_1\in E(G)}\ge \floor{r/2}+1.
\]
Also $d_H(v_1)=\floor{r/2}+1$. Hence $d_H(u)\ge \floor{r/2}+1$ for all $u\in A\cup\{v_1\}$, which proves $(ii)$.

Furthermore, we have $2e(H)=ra-e_G(A,B)+2\floor{r/2}+2$. Then~\eqref{eq:cut-lower} gives
\begin{align*}
    2e(H)\le (2\ceil{r/2}-2)b-\ceil{r/2}(\ceil{r/2}-1)+2\floor{r/2}+2.
\end{align*}
Thus, for $r=5,8$, we have $2e(H)\le (\floor{r/2}+2)b$, which proves $(i)$ together with $b\le a+1$.

Now we consider $(iii)$. Since $G[B]$ is $(\ceil{r/2}-1)$-degenerate, we may continue the ordering from the chosen $v_1$ by successively deleting a vertex of degree at most $\ceil{r/2}-1$ in the remaining graph on $B$.
For $0\le i\le b$, set $B_i=B\setminus\{v_1,\ldots,v_i\}$ and $A_i=A\cup\{v_1,\ldots,v_i\}$ with $B_0=B$ and $A_0=A$. Thus
$d_{B_{i-1}}(v_i)\le \ceil{r/2}-1$.
Moving $y_i$ from $B_{i-1}$ to $A_{i-1}$ changes the cut size by
$2d_{B_{i-1}}(v_i)-r\le 2\ceil{r/2}-r-2$. Consequently,
\[
 e_G(A_i,B_i)\le e_G(A,B)-(2\ceil{r/2}-r-2)i.
\]
Fix $1\le t\le b$. Now take any $S\subseteq A_t$ with $|S|=t$ and move all vertices of
$S$ to the other part.
Since, after vertices removing, $A_t\setminus S$ has exactly $a$ vertices and the another part has exactly $b$ vertices, the resulting parts is also a bisection.
Moreover, the change in the cut size is exactly
\[
 2\sum_{v\in S}d_{A_t}(v)-rt-2e_G(S).
\]
By the minimality of $(A,B)$,
\begin{align*}
-(2\ceil{r/2}-r-2)t+2\sum_{v\in S}d_{A_t}(v)-rt-2e_G(S)\ge0.
\end{align*}
Therefore
\begin{equation}\label{eq:multi-exchange}
 e_G(S)\le\sum_{v\in S}(d_{A_t}(v)-\ceil{r/2}+1).
\end{equation}
Fix $\ell\le b$. Since every $v_i$ has degree $r$, we have
$|X_\ell|\le r(\ell-1)$. 
Let $S\subseteq V(H)\setminus X_\ell$ be a nonempty set with $t=|S|\le \ell$. Then no vertex of $S$ is adjacent to any of $v_2,\ldots,v_t$. 
Hence $d_{A_t}(v)=d_H(v)$ for every $v\in S$ and $e_G(S)=e_H(S)$. Conclusion $(iii)$ now follows from~\eqref{eq:multi-exchange}.
\end{proof}

\section{Small dense subgraphs}
\label{sec:local}
In this section, we mainly consider small vertex sets spanning more edges than vertices in special graphs with bounded degrees.

\begin{lemma}\label{lem:two-rounds}
Let $p\ge3$ be an integer, and let $H$ be a graph on $N$ vertices such that $2e(H)\le(p+1)N$ and $p\le d_H(v)\le p+3$ for $v\in V(H)$.
Suppose that $X$ is a subset of $V(H)$ such that for $S\subseteq V(H)\setminus X$ with $|S|\le17$,
\begin{equation}\label{eq:small-local}
 e_H(S)\le\sum_{v\in S}(d_H(v)-p).
\end{equation}
Let $D=\{v\in V(H):d_H(v)=p+1\}$. Then
\begin{equation}\label{eq:degree-surplus}
 e_H(D)-|D|
 \ge \left(\frac{p-1}{2}-\frac2p\right)N
       -4(p+1)|X|.
\end{equation}
\end{lemma}

\begin{proof}[\bf Proof.]
Put $w(v)=d_H(v)-p\in\{0,1,2,3\}$.
Let
\[
 P_0=\{v\in V(H)\setminus X:d_H(v)=p\},
 \qquad H_1=H-P_0.
\]
Note that $P_0$ is independent, since an edge between two of its vertices
would violate~\eqref{eq:small-local} on a two-vertex set of total weight zero.
We claim that for $v\in V(H_1)\setminus X$,
\begin{equation}\label{eq:first-residual-degree}
 d_{H_1}(v)\ge p.
\end{equation}
Otherwise such a vertex $v$ would have at least $w(v)+1$ neighbours
in $P_0$. Taking $v$ together with $w(v)+1$ of these neighbours gives
a set of at most five vertices with at least $w(v)+1$ edges and total
weight $w(v)$, contrary to~\eqref{eq:small-local}.

Let
\[
 P_1=\{v\in V(H_1)\setminus X:d_{H_1}(v)=p\},\text{ and }C_v=\{v\}\cup(N_H(v)\cap P_0)\text{ for $v\in P_1$}.
\]
Note that, for some vertex $v$ of $V(H_1)\setminus X$, $v\in P_1$ if and only if it has exacly $w(v)$ neighbours in $P_0$. 
Thus $|N_H(v)\cap P_0|=w(v)$ and $|C_v|\le4$. Furthermore, by the independence of $P_0$, we have $e_H(C_v)=\sum_{u\in C_v}w(u)=w(v)$.
We also claim that $P_1$ is independent. Indeed, if distinct $x,y\in P_1$
were adjacent, then 
\[ 
    e_H(C_x\cup C_y)
    = w(x)+w(y)+1
    >\sum_{u\in C_x\cup C_y}w(u),
\]
which, together with $|C_x\cup C_y|\le8$, contradicts~\eqref{eq:small-local}.

Set $H_2=H-(P_0\cup P_1)$. We next claim that, for $v\in V(H_2)\setminus X$,
\begin{equation}\label{eq:second-residual-degree}
 d_{H_2}(v)\ge p.
\end{equation}
Suppose otherwise. Choose $w(v)+1\le4$ distinct neighbours of $v$ in $P_0\cup P_1$. For a chosen neighbour $u\in P_0$, take $\{u\}$; for a chosen neighbour $u\in P_1$, take $C_u$.
Let $U$ be the union of the sets so obtained.
Note that every vertex $u$ of $U\cap P_1$ is one of the chosen centres, and all of its $w(u)$ neighbours in $P_0$ belong to $U$. In the other hand, all other vertices of $U$ belong to $P_0$ and have weight zero. Hence $e_H(U)=\sum_{u\in U}w(u)$.
Moreover, the vertex $v$ does not belong to $U$ and has at least $w(v)+1$ neighbours in $U$. Therefore
\[
 e_H(U\cup\{v\})
 \ge\sum_{u\in U}w(u)+w(v)+1
 >\sum_{u\in U\cup\{v\}}w(u).
\]
But $U\cup\{v\}$ is disjoint from $X$ and has at most
$1+4\cdot4=17$ vertices. This contradicts \eqref{eq:small-local} and proves
\eqref{eq:second-residual-degree}.

For $j\in\{0,1,2,3\}$, let $n_j=|\{v\in V(H):d_H(v)=p+j\}|$.
Since $2e(H)\le(p+1)N$, we have
\begin{equation}\label{eq:degree-balance}
 n_0\ge n_2+2n_3.
\end{equation}
Since $P_0$ is independent and each of its vertices has degree $p$,
we have $e(H_1)=e(H)-p|P_0|$. Thus
\begin{align*}
 (p+1)|V(H_1)|-2e(H_1)= (p+1)N-2e(H)+(p-1)|P_0|\ge(p-1)|P_0|.
\end{align*}
On the other hand, $(p+1)|V(H_1)|-2e(H_1)=\sum_{v\in V(H_1)}(p+1-d_{H_1}(v))$. By
\eqref{eq:first-residual-degree}, each vertex outside $X$ contributes
$1$ if it belongs to $P_1$, and at most $0$ otherwise. Each vertex of
$X$ contributes at most $p+1$. Hence
\[
 |P_1|\ge(p-1)|P_0|-(p+1)|X|.
\]
Writing $q=|P_0|+|P_1|$ and using $|P_0|\ge n_0-|X|$, we obtain
\begin{equation}\label{eq:deleted-lower}
 q\ge p n_0-(2p+1)|X|.
\end{equation}
In each deletion round $i\in\{1,2\}$, $P_i$ is independent and each of its vertices has degree exactly $p$ in the current graph. Therefore $e(H_2)=e(H)-pq\le (p+1)N/2-pq$.
Moreover, none of the vertices of $X$ has been deleted. By
\eqref{eq:second-residual-degree},
\[
 (p+1)N-2pq
 \ge2e(H_2)
 \ge p(N-q-|X|).
\]
It follows that $q\le N/p+|X|$. Together with
\eqref{eq:deleted-lower}, this gives
\begin{equation}\label{eq:low-degree-count}
 n_0\le\frac{N}{p^2}+\frac{2(p+1)}p\,|X|.
\end{equation}

Finally, $|D|=n_1$, and the number of edges from $D$ to its complement
is at most $pn_0+(p+2)n_2+(p+3)n_3$. Consequently,
\begin{align*}
 e_H(D)-|D|
 &\ge\frac{(p+1)n_1-pn_0-(p+2)n_2-(p+3)n_3}{2}-n_1\\
 &=\frac{p-1}{2}N-\frac{2p-1}{2}n_0
                    -\frac{2p+1}{2}n_2-\frac{2p+2}{2}n_3\\
 &\ge\frac{p-1}{2}N-2p n_0,
\end{align*}
where the second equality holds by $N=n_0+n_1+n_2+n_3$ and the last inequality holds since
\[
 \frac{2p+1}{2}n_2+\frac{2p+2}{2}n_3
 \le\frac{2p+1}{2}(n_2+2n_3)
 \le\frac{2p+1}{2}n_0
\]
by~\eqref{eq:degree-balance}. Together with \eqref{eq:low-degree-count}, this proves~\eqref{eq:degree-surplus}.
\end{proof}

We also need a small set whose induced graph has more edges than
vertices. 

\begin{lemma}\label{lem:small-dense-set}
Let $d\ge2$ and $\ell\ge2$ be integers. Let $F$ be a graph on at
most $N$ vertices with maximum degree at most $d$. If $e(F)-|V(F)|>N^2/\ell^2,$ then there is a set $S\subseteq V(F)$ such that
$|S|\le2d\ell$ and $e_F(S)>|S|$.
\end{lemma}

\begin{proof}
Suppose that no such set exists. Every connected component of $F$
with fewer than $\ell$ vertices has at most as many edges as vertices.
Delete all such components, and call the resulting graph $F'$. Then
\begin{equation}\label{eq:component-deletion}
 e(F')-|V(F')|\ge e(F)-|V(F)|>\frac{N^2}{\ell^2}.
\end{equation}

We partition each connected component of $F'$ into connected vertex
sets of orders between $\ell$ and $d\ell$. To justify this, take a
spanning tree of the component and root it at a leaf. Every vertex has
at most $d-1$ children. As long as the remaining rooted tree has at
least $\ell$ vertices, choose a deepest vertex whose descendant subtree
has at least $\ell$ vertices, and remove that subtree. Every child
subtree of the chosen vertex has fewer than $\ell$ vertices, so the
removed set has order at most $1+(d-1)(\ell-1)$.
The remaining vertices, if any, still form a rooted tree. If the final
remainder is nonempty and has fewer than $\ell$ vertices, merge it with
one of the previously removed sets adjacent to it in the original
spanning tree. Such a set exists by connectedness. The merged set is
connected and has order at most
\[
 (\ell-1)+1+(d-1)(\ell-1)
 =d\ell-d+1\le d\ell.
\]
At least one set was removed from each component because its order was
at least $\ell$. Thus the construction gives the claimed partition,
including when the final remainder is empty.

Let the resulting sets be $C_1,\ldots,C_q$. They are disjoint, cover
$V(F')$, and satisfy $\ell\le|C_i|\le d\ell$, so $q\le N/\ell$.
Our supposition implies $e_{F'}(C_i)\le|C_i|$ for every $i$.
Moreover, at most two edges join any two different sets $C_i,C_j$.
Otherwise, since each set is connected,
\[
 e_{F'}(C_i\cup C_j)
 \ge(|C_i|-1)+(|C_j|-1)+3
 =|C_i\cup C_j|+1,
\]
contrary to the supposition, as $|C_i\cup C_j|\le2d\ell$.
Consequently,
\[
 e(F')-|V(F')|
 \le2\binom q2\le q^2\le\frac{N^2}{\ell^2},
\]
contradicting~\eqref{eq:component-deletion}.
\end{proof}

\section{Proof of the main results}
\label{sec:main-proof}
In this section, we give a proof of Theorem \ref{thm:structure}, which together  with  Lemma \ref{lem:extension} implies Theorem \ref{thm:main}.

\begin{proof}[\normalfont\bfseries Proof of Theorem~\ref{thm:structure}]
Let $G$ have order $n\ge10^7$, and fix an arbitrary minimum
bisection $(A,B)$. Set $p=\floor{r/2}+1$. Suppose that $B$ contains no $p$-cohesive set.

Apply Lemma~\ref{lem:exchange}, with the notation used there.
In particular, $H=G[A\cup\{y_1\}]$ and $N=|A|+1$. Since
$|\,|A|-|B|\,|\le1$, we have $N\ge n/2\ge5\cdot10^6$ and $b\ge N-2$. First apply~\eqref{eq:local-exchange} with $L=17$. Let $D=\{v\in V(H):d_H(v)=p+1\}$. Since
$|X_{17}|\le16r$, Lemma~\ref{lem:two-rounds} gives
\begin{equation}\label{eq:initial-surplus}
e_H(D)-|D|
\ge
\left(\frac{p-1}{2}-\frac2p\right)N
-64r(p+1).
\end{equation}

Set $s=\ceil{\sqrt N}$ and $T=24s$. As $N\ge5\cdot10^6$,
we have $s\le\sqrt N+1\le N/2000+1$. Consequently, $T\le\frac{3N}{250}+24\le N-2\le b.$ Thus $X_T$ is defined. Let $F=H[D\setminus X_T]$. Then
$|V(F)|\le N$ and $\Delta(F)\le p+1$.
Since $|X_T|\le r(T-1)$, deleting $D\cap X_T$ decreases
the number of edges minus the number of vertices by at most
$p|D\cap X_T|\le prT$. Hence~\eqref{eq:initial-surplus}
implies
\begin{equation}\label{eq:remaining-surplus}
e(F)-|V(F)|
\ge
\left(\frac{p-1}{2}-\frac2p\right)N
-64r(p+1)-24prs.
\end{equation}

We claim that
\[
e(F)-|V(F)|>\frac{(p+1)^2N}{144}.
\]
Indeed, for $r\in\{5,8\}$ and $p=\floor{r/2}+1$,
the following numerical inequalities hold:
\[
\frac{p-1}{2}-\frac2p-\frac{(p+1)^2}{144}
\ge\frac{2pr}{135},
\qquad
64r(p+1)\le86pr.
\]
It follows from~\eqref{eq:remaining-surplus} that
\begin{align*}
e(F)-|V(F)|-\frac{(p+1)^2N}{144}
&\ge pr\left(\frac{2N}{135}-24s-86\right)\\
&\ge pr\left(
\left(\frac2{135}-\frac3{250}\right)N-110
\right)>0,
\end{align*}
where the last inequality follows from $N\ge5\cdot10^6$.

Now take $\ell=12s/(p+1)$, which is an integer for
$p\in\{3,5\}$. Since $s^2\ge N$,
\[
e(F)-|V(F)|
>\frac{(p+1)^2N}{144}
\ge\frac{N^2}{\ell^2}.
\]
Applying Lemma~\ref{lem:small-dense-set} with $d=p+1$,
we obtain a set $S\subseteq V(F)$ such that
\[
e_F(S)>|S|,
\qquad
|S|\le2(p+1)\ell=24s=T.
\]
Since $S\subseteq D\setminus X_T$, every vertex of $S$
has degree exactly $p+1$ in $H$. Applying
\eqref{eq:local-exchange} with $L=T$ yields
\[
e_F(S)=e_H(S)
\le\sum_{v\in S}(d_H(v)-p)
=|S|,
\]
a contradiction. Thus both parts of the chosen minimum bisection contain
a $p$-cohesive set. Since the minimum bisection is arbitrary, the theorem follows.
\end{proof}

\begin{proof}[\normalfont\bfseries Proof of Theorem~\ref{thm:main}]
Choose a minimum bisection of $G$. By
Theorem~\ref{thm:structure}, its two parts contain disjoint
$\ceil{r/2}$-cohesive sets. Lemma~\ref{lem:extension} extends them to
an internal partition of $G$.
\end{proof}

\section*{Declaration on the Use of AI}

The authors used ChatGPT 5.6 to assist in discussing proof strategies, checking proofs, and improving exposition.


\begin{thebibliography}{9}

\bibitem{BL2016}
A. Ban~and N.~Linial,
\textit{Internal partitions of regular graphs},
J. Graph Theory \textbf{83} (2016) 5--18.


\bibitem{BNP2024}
P.~B\"arnkopf, Z.~L.~Nagy and Z.~Paulovics,
\textit{A note on internal partitions: the $5$-regular case and beyond},
Graphs Combin. \textbf{40} (2024) 36.

\bibitem{BTV2003}
C.~Bazgan, Z.~Tuza and D.~Vanderpooten,
\textit{On the existence and determination of satisfactory partitions in a graph},
in {\it Algorithms and computation}, Springer, Berlin, Heidelberg, 2003, 444--453.

\bibitem{BTV2006}
C.~Bazgan, Z.~Tuza and D.~Vanderpooten,
\textit{The satisfactory partition problem},
Discrete Appl. Math. {\bf 154} (2006) 1236--1245.


\bibitem{DeVos2009}
M.~DeVos,
\textit{Friendly partitions},
Open Problem Garden (2009),
\url{http://www.openproblemgarden.org/op/friendly_partitions}.


\bibitem{Diwan2000}
A.~A. Diwan,
\textit{Decomposing graphs with girth at least five under degree constraints},
J. Graph Theory {\bf 33} (2000) 237--239


\bibitem{GK2000}
M.~U.~Gerber and D.~Kobler,
\textit{Algorithmic approach to the satisfactory graph partitioning problem},
European J. Oper. Res. \textbf{125} (2000) 283--291.

\bibitem{Kaneko1998}
A. Kaneko,
\textit{On decomposition of triangle free graphs under degree constraints},
J. Graph Theory \textbf{27} (1998) 7--9.


\bibitem{MW2024}
J. Ma and H. Wu,
\textit{Partitioning graphs with linear minimum degree},
Random Struct. Alg. {\bf 65} (2024) 601--609

\bibitem{MY2019}
J. Ma and T. Yang,
\textit{Decomposing $C_4$-free graphs under degree constraints},
J. Graph Theory {\bf 90} (2019) 13--23

\bibitem{SD2002}
K.~H.~Shafique and R.~D.~Dutton,
\textit{On satisfactory partitioning of graphs},
Congr. Numer. \textbf{154} (2002) 183--194.

\bibitem{Stiebitz1996}
M. Stiebitz,
\textit{Decomposing graphs under degree constraint},
J. Graph Theory \textbf{23} (1996) 321--324.

\bibitem{Thomassen1983}
C.~Thomassen,
\textit{Graph decomposition with constraints on the connectivity and minimum degree},
J. Graph Theory \textbf{7} (1983) 165--167.

\bibitem{WW2024}
Y.~Wang and H.~Wu,
\textit{Graph partitions under average degree constraint},
J. Combin. Theory Ser. B \textbf{165} (2024) 197--210.


\end{thebibliography}
\end{document}